\documentclass[11pt]{amsart}
\usepackage[british]{babel}

\usepackage{a4wide}
\usepackage{comment}
\usepackage{titletoc}
\usepackage{amsfonts, amsmath, amsthm, amssymb, fancybox, graphicx, color, times, babel, latexsym}
\usepackage[utf8]{inputenc}
\usepackage{multirow}
\usepackage[usenames,dvipsnames]{xcolor}
\usepackage[all,cmtip]{xy}
\usepackage[normalem]{ulem}

\newtheorem {thm}{Theorem}
\newtheorem* {thm*}{Theorem}

\newtheorem {cor}[thm]{Corollary}

\newtheorem* {cor*}{Corollary}
\newtheorem {lem}[thm]{Lemma}

\newtheorem {prop}[thm]{Proposition}

\theoremstyle{definition}
\newtheorem {defi}[thm]{Definition}
\newtheorem {rem}[thm]{Remark}

\newtheorem* {conj*}{Conjecture}

\newtheorem* {quest*}{Question}

\DeclareMathOperator{\Gal}{Gal}

\DeclareMathOperator{\Aut}{Aut}

\DeclareMathOperator{\Dens}{Dens}
\DeclareMathOperator{\Res}{Res}

\newcommand{\F}{\mathbb{F}}

\newcommand{\Q}{\mathbb{Q}}
\newcommand{\Z}{\mathbb{Z}}
\renewcommand{\F}{\mathbb{F}}

\newcommand{\Qbar}{\overline{\mathbb{Q}}}
\newcommand{\Zhat}{\widehat{\mathbb{Z}}}

\newcommand\C{\mathbb{C}}

\newcommand{\sqf}{\mathrm{sf}}

\DeclareMathOperator{\GL}{GL}
\DeclareMathOperator{\SL}{SL}
\DeclareMathOperator{\im}{Im}

\makeatletter
\newcommand{\customlabel}[2]{%
   \protected@write \@auxout {}{\string \newlabel {#1}{{#2}{\thepage}{#2}{#1}{}} }%
   \hypertarget{#1}{}
}
\makeatother

\usepackage{xr-hyper}
\usepackage{hyperref}

\newcommand{\imGal}{\mathcal{G}}

\newcommand{\fine}{\mathrm{F}}

\newcommand{\failure}{\mathrm{C}}

\newcommand{\set}{\mathcal{W}}
\newcommand{\function}{\mathrm{w}}

\title{Reductions of points on algebraic groups, II}
\author{Peter Bruin and Antonella Perucca}
\address[]{Universiteit Leiden, Mathematisch Instituut,
Postbus 9512, 2300 RA Leiden, The Netherlands}
\email[]{P.J.Bruin@math.leidenuniv.nl}
\address[]{University of Luxembourg. 6, avenue de la Fonte, 4364 Esch-sur-Alzette, Luxembourg}
\email[]{antonella.perucca@uni.lu}

\begin{document}

\begin{abstract} Let $A$ be the product of an abelian variety and a torus over a number field $K$, and let $m$ be a positive integer. If $\alpha \in A(K)$ is a point of infinite order, we consider the set of primes $\mathfrak p$ of $K$ such that the reduction $(\alpha \bmod \mathfrak p)$ is well defined and has order coprime to $m$. This set admits a natural density, which we are able to express as a finite sum of products of $\ell$-adic integrals, where $\ell$ varies in the set of prime divisors of $m$. We deduce that the density is a rational number, whose denominator is bounded (up to powers of $m$) in a very strong sense. This extends the results of the paper \emph{Reductions of points on algebraic groups} by Davide Lombardo and the second author, where the case $m$ prime is established.
\end{abstract}

\maketitle

\section{Introduction}

This article is the continuation of the paper \emph{Reductions of points on algebraic groups} by Davide Lombardo and the second author \cite{LombardoPerucca}. We refer to this other work for the history of the problem and further references.

Let $A$ be the product of an abelian variety and a torus over a number field $K$, and let $m$ be a positive square-free integer. If $\alpha \in A(K)$ is a point of infinite order, we consider the set of primes $\mathfrak p$ of $K$ such that the reduction $(\alpha \bmod \mathfrak p)$ is well defined and has order coprime to $m$. This set admits a natural density (see Theorem \ref{thm:interpretation}), which we denote by $\Dens_m(\alpha)$.

The main question is whether we can write
\begin{equation}\label{question}
\Dens_m(\alpha)=\prod_{\ell\mid m} \Dens_\ell(\alpha)
\end{equation}
where $\ell$ varies over the prime divisors of $m$.
Let $K(A[m])$ be the $m$-torsion field of~$A$. We prove that \eqref{question} holds if $K(A[m])=K$ (i.e. if $A(K)$ contains all $m$-torsion points) or, more generally, if the degree $[K(A[\ell]):K]$ is a power of $\ell$ (see Corollary \ref{corproduct}). Indeed, \eqref{question} holds if the torsion/Kummer extensions of $\alpha$ related to different prime divisors of $m$ are linearly disjoint over $K$. In general, \eqref{question} does not hold; see Section \ref{serre} for an explicit example.

We are able to express $\Dens_m(\alpha)$ as an integral over the image of the $m$-adic representation (see Theorem \ref{madicintegral}), and also as a finite sum of products of $\ell$-adic integrals (see Theorems \ref{thmmany} and~\ref{thm:densityproduct}).
The latter decomposition allows us to prove that $\Dens_m(\alpha)$ is a rational number whose denominator is uniformly bounded in a very strong sense (see Corollary \ref{wunder}).

Finally, we study Serre curves in detail in Section \ref{serrecurves}. With our  results one can explicitly compute $\Dens_m(\alpha)$ if the $m^n$-Kummer extensions of $\alpha$ (defined in Section~\ref{arboreal}) have maximal degree for all $n$ or, more generally, if the degrees of these extensions are known and are the same with respect to the base fields $K$ and $K(A[m])$.

\section{Integration on profinite groups}

For every compact topological group $G$, we write $\mu_G$ for the normalised Haar measure on~$G$.

Let $G$ be a profinite group, and let $H$ be a closed subgroup of finite index in~$G$.
For every integrable function $f\colon H\to\C$, we define
\begin{equation}\label{int}
I_{H,f} := \int_H f d\mu_H\,.
\end{equation}

Let $\ell$ vary in a finite set of prime numbers, and suppose that we have $G=\prod_{\ell}G_\ell$, where each $G_\ell$ is
a profinite group containing a pro-$\ell$-group $G'_\ell$ as a closed subgroup of finite index. Note that we may assume that $G'_\ell$
is normal in $G_\ell$ (up to replacing $G'_\ell$ by the intersection of its finitely many conjugates in $G_\ell$).

The profinite group $G':=\prod_{\ell}G'_\ell$ has finite index in~$G$, and the profinite group
\[
H' := H \cap G'
\]
is a closed subgroup of finite index in each of $G$, $H$ and~$G'$.
Because the $G'_\ell$ are pro-$\ell$-groups for pairwise different $\ell$, every
closed subgroup of~$G'$ is similarly a product of pro-$\ell$-groups.  We
can therefore write
\begin{equation}
  H' = \prod_{\ell}H'_\ell
  \label{eq:H^1-product}
\end{equation}
where each $H'_\ell$ is a closed subgroup of finite index in~$G'_\ell$.  The normalised Haar measures on $H'$ and the $H'_\ell$ are thus
related by
\[
\mu_{H'} = \prod_{\ell}\mu_{H'_\ell}\,.
\]

For each $x\in H/H'$, we write $H(x)$ for the fibre over~$x$ of the
quotient map $H\to H/H'$.  Then each $H(x)$ is a left coset of the
subgroup $H'\subseteq H$.  We restrict the Haar measure on~$H$ to a
measure on the open subset $H(x)$, which we still denote by $\mu_H$.
The normalised Haar measure on $H(x)$ is then
\[
\mu_{H(x)} = (H:H')\mu_H\,,
\]
and we can decompose \eqref{int} as
\begin{equation}\label{integralstep}
  I_{H,f} = \sum_{x\in H/H'} \int_{H(x)} f d\mu_H= \frac{1}{(H:H')}\sum_{x\in H/H'} \int_{H(x)} f d\mu_{H(x)}.
\end{equation}

By \eqref{eq:H^1-product}, and since $H(x)$ is a torsor under $H'$, we
can write
\[
H(x) = \prod_{\ell} H_\ell(x)\,,
\]
where each $H_\ell(x)$ is a torsor under $H'_\ell$.  We equip each
$H_\ell(x)$ with the normalised Haar measure $\mu_{H_\ell(x)}$ coming from
the $H'_\ell$-torsor structure; then $\mu_{H(x)}$ is the product of the
$\mu_{H_\ell(x)}$.

\begin{thm}\label{thm:decomposition}
Suppose that over $H(x)$ there is a product decomposition
\begin{equation}\label{decomposition}
  f = \prod_{\ell} f_{x,\ell}
\end{equation}
where the $f_{x,\ell}: H_\ell(x)\to\C$ are locally constant functions that are integrable with respect to $\mu_{H_\ell(x)}$. We then have
\begin{equation}\label{formula1}
I_{H,f} = \frac{1}{(H:H')}\sum_{x\in H/H'} \prod_{\ell}
\int_{H_\ell(x)} f_{x,\ell} d\mu_{H_\ell(x)}.
\end{equation}
\end{thm}
\begin{proof}
The assertion follows from \eqref{integralstep} because we have a product decomposition for $f$ and because $\mu_{H(x)}$ is the product of the
$\mu_{H_\ell(x)}$.
\end{proof}

\begin{prop}\label{prop:integral}
Suppose that we have $H=\prod_{\ell} H_\ell$, where $H_\ell\subseteq G_\ell$ (this happens for example if $H\subseteq G'$).
If there is a product decomposition
\begin{equation}
  f= \prod_{\ell} f_{\ell}
  \label{eq:f-product-all}
\end{equation}
over $H$, where the $f_{\ell}: H_\ell\to\C$ are locally constant functions that are integrable with respect to $\mu_{H_\ell}$, then we have
\begin{equation}
I_{H,f} = \prod_{\ell}
\int_{H_\ell} f_\ell d\mu_{H_\ell}\,.
\end{equation}
\end{prop}
\begin{proof}
The assertion follows from the product decomposition for $f$ and the fact that $\mu_{H}$ is the product of the $\mu_{H_\ell}$.
\end{proof}

\section{The arboreal representation}\label{arboreal}

Let $K$ be a number field, and let $\overline{K}$ be an algebraic closure of~$K$. Let $A$ be a connected commutative algebraic group over $K$, and let $b_A$ be the first Betti number of $A$. We fix a square-free positive integer $m$, and we let $\ell$ vary in the set of prime divisors of $m$. We also fix a point $\alpha\in A(K)$.

We define $T_m A$ as the projective limit of the torsion groups $A[m^n]$ for $n\geqslant 1$; we can write $T_m A=\prod_{\ell} T_\ell A$, where the Tate module $T_\ell A$ is a free $\Z_\ell$-module of rank~$b_A$.

We define the \emph{torsion fields}
$$
K_{m^{-n}}:=K(A[m^n])\quad\text{for }n\geqslant1
$$
and
$$
K_{m^{-\infty}} := \bigcup_{n\geqslant1}K_{m^{-n}}.
$$
The Galois action on the $m$-power torsion points of $A$ gives the \emph{$m$-adic representation} of $A$, which maps $\Gal(\overline{K}/K)$ to the automorphism group of $T_m A$. We can also speak of the \emph{mod $m^n$ representation}, which describes the Galois action on $A[m^n]$. Choosing a $\Z_\ell$-basis for $T_\ell A$, we can identify the image of the $m$-adic representation with a subgroup of $\prod_\ell \GL_{b_A}(\Z_\ell)$ and the image of the mod $m^n$ representation with a subgroup of $\prod_\ell \GL_{b_A}(\Z/\ell^n\Z)$.

For $n\geqslant 1$, let $m^{-n}\alpha$ be the set of points in $A(\overline{K})$ whose $m^n$-th multiple equals $\alpha$.
We also write
$$
m^{-\infty}\alpha = \varprojlim_{n\geqslant 1} m^{-n}\alpha\,.
$$
This is the set of sequences $\beta=\{\beta_n\}_{n\geqslant 1}$ such that $m\beta_1=\alpha$ and $m\beta_{n+1}=\beta_n$ for every $n\geqslant 1$; it is a torsor under $T_m A$.  We note that $m^{-n}0=A[m^n]$ and $m^{-\infty}0=T_m A$.

We define the fields
$$
K_{m^{-n}\alpha}:=K(m^{-n}\alpha)\quad\text{for }n\geqslant1
$$
and
$$
K_{m^{-\infty}\alpha} := \bigcup_{n\geqslant1}K_{m^{-n}\alpha}.
$$
We call the field extension $K_{m^{-n}\alpha}/K_{m^{-n}}$ the $m^n$-\emph{Kummer extension} defined by the point~$\alpha$.
We view the $m$-adic representation as a representation of $\Gal({K_{m^{-\infty}\alpha}}/K)$.

We fix an element $\beta\in m^{-\infty}\alpha$, and define the \emph{arboreal representation}
$$
\begin{array}{cccc}
\omega_\alpha : & \Gal(K_{m^{-\infty}\alpha}/K) & \longrightarrow & T_m A \rtimes \Aut(T_m A)\\
& \sigma & \longmapsto & (t, M)
\end{array}
$$
where $M$ is the image of $\sigma$ under the $m$-adic representation and
$t = \sigma(\beta)-\beta$. The arboreal representation is an injective homomorphism of profinite groups identifying $\Gal(K_{m^{-\infty}\alpha}/K)$ with a subgroup of $$T_m A \rtimes \Aut(T_m A)\cong  \prod_\ell \Z_\ell^{b_A} \rtimes \prod_\ell \GL_{b_A}(\Z_\ell)\cong\prod_\ell (\Z_\ell^{b_A} \rtimes \GL_{b_A}(\Z_\ell))\,.$$

Similarly, for $\sigma \in \Gal(K_{m^{-n}\alpha}/K)$ we can consider $t \in A[m^n] \cong \prod_\ell(\Z/\ell^n\Z)^{b_A}$ and $M \in \Aut A[m^n] \cong \prod_\ell \GL_{b_A}(\Z/\ell^n\Z)$.

Denote by $\imGal({m^\infty}) \subseteq \prod_\ell \GL_{b_A}(\Z_\ell)$ the image of the $m$-adic representation and by $\imGal({m^n})$ the image of the mod $m^n$ representation. Consider the image of the $\ell$-adic representation in $\GL_{b_A}(\Z_\ell)$: the dimension of its Zariski closure in $\GL_{b_A,\mathbb{Q}_\ell}$ is independent of $\ell$ and $K$, and we denote it by $d_A$. For example, if $A$ is an elliptic curve, we have $d_A=2$ if $A$ has  complex multiplication, and $d_A=4$ otherwise.

\begin{defi}\label{defi-conditions} We say that $(A/K,m)$ satisfies \emph{eventual maximal growth of the torsion fields} if there exists a positive integer $n_0$ such that for all $N\geqslant n\geqslant n_0$ we have
$$
[K_{m^{-N}}:K_{m^{-n}}]=m^{d_A(N-n)}\,.
$$
We say that $(A/K,m, \alpha)$ satisfies \emph{eventual maximal growth of the Kummer extensions} if there exists a positive integer $n_0$ such that for all $N\geqslant n \geqslant n_0$ we have
\begin{equation}\label{refomulation-Defi}
[K_{m^{-N}\alpha}:K_{m^{-n}\alpha}]=(m^{b_A+d_A})^{N-n}\,.
\end{equation}
\end{defi}

\begin{rem}
Condition \eqref{refomulation-Defi} means that there is eventual maximal growth of the torsion fields, that $K_{m^{-n}\alpha}$ and $K_{m^{-N}}$ are linearly disjoint over $K_{m^{-n}}$, and that we have $$[K_{m^{-N}\alpha}:K_{m^{-N}}(m^{-n}\alpha)]=m^{b_A(N-n)}\,.$$
If there is eventual maximal growth of the Kummer extensions, the rational number
\begin{equation}\label{defi-failure}
\failure_m:=m^{b_A n}/[K_{m^{-n}\alpha}:K_{m^{-n}}]
\end{equation}
is independent of $n$  for $n \geqslant n_0$.
\end{rem}

\begin{prop}\label{prop-maximalgrowth} If $A$ is a semiabelian variety, then $(A/K,m)$ satisfies eventual maximal growth of the torsion fields. If $A$ is the product of an abelian variety and a torus and $\mathbb Z \alpha$ is Zariski dense in $A$, then $(A/K,m,\alpha)$ satisfies eventual maximal growth of the Kummer extensions.
\end{prop}
\begin{proof}
By \cite[Lemma 10]{LombardoPerucca}, if $A$ is a semiabelian variety and $\ell$ is a prime divisor of~$m$, then $(A/K,\ell, \alpha)$ satisfies eventual maximal growth of the torsion fields. We also know that the degree $[K_{\ell^{-n}}:K_{\ell^{-1}}]$ is a power of $\ell$. Therefore the extensions $K_{m^{-1}}K_{\ell^{-n}}$ for $\ell\mid n$ are linearly disjoint over $K_{m^{-1}}$ and the first assertion follows.
By \cite[Remark 7]{LombardoPerucca}, the second assertion holds for $(A/K,\ell, \alpha)$, where $\ell$ is any prime divisor of $m$. We conclude because these Kummer extensions have order a power of $\ell$.
\end{proof}

\section{Relating the density and the arboreal representation}

\subsection{The existence of the density}

From now on, we assume that $(A/K,m,\alpha)$ satisfies eventual maximal growth of the Kummer extensions.

\begin{rem}
This is not a restriction if $A$ is the product of an abelian variety and a torus by Proposition \ref{prop-maximalgrowth}. Indeed, consider the number of connected components of the Zariski closure of $\mathbb Z \alpha$. If this number is not coprime to $m$, then the density $\Dens_m(\alpha)$ is zero by \cite[Main Theorem]{PeruccaJNT} while if it is coprime to $m$ we may replace $\alpha$ by a multiple to reduce to the case where the Zariski closure of $\mathbb Z \alpha$ is connected. Finally, we may replace $A$ and reduce to the case where $\mathbb Z \alpha$ is Zariski dense.
\end{rem}

The $T_m A$-torsor $m^{-\infty}\alpha$ from Section \ref{arboreal} defines a Galois cohomology class
$$
C_\alpha \in H^1(\Gal(K_{m^{-\infty}\alpha}/K), T_m A)\,.
$$
For any choice of $\beta\in m^{-\infty}\alpha$, this is the class of the cocycle
$$
\begin{array}{cccc}
c_\beta: & \Gal(K_{m^{-\infty}\alpha}/K)& \longrightarrow & T_m A\\
& \sigma & \longmapsto & \sigma(\beta)-\beta\,.
\end{array}$$
We also consider the restriction map with respect to the cyclic subgroup generated by some element $\sigma$:
$$\Res_\sigma:  H^1\bigl(\Gal(K_{m^{-\infty}\alpha}/K), T_m A\bigr) \longrightarrow H^1({\langle \sigma \rangle}, T_m A)\,.$$

\begin{thm}\label{thm:interpretation}
If $(A/K,m,\alpha)$ satisfies eventual maximal growth of the Kummer extensions, then the density $\Dens_m(\alpha)$ exists and equals the normalized Haar measure in $\Gal(K_{m^{-\infty}\alpha}/K)$ of the subset
$$
S_\alpha:=\{\sigma\,:\, C_\alpha\in \ker(\Res_\sigma) \}=\{\sigma\;:\;  \sigma(\beta)=\beta\, \text{ for some $\beta\in m^{-\infty} \alpha$}\}\,.
$$
\end{thm}
\begin{proof}
The generalisations of \cite[Theorem 3.2]{JonesRouse} and \cite[Theorem 5]{LombardoPerucca} to the composite case are straightforward.
\end{proof}

Similarly to \cite[Remark 19]{LombardoPerucca}, we may equivalently consider $S_\alpha$ as a subset of either $\Gal(\overline{K}/K)$ or $\Gal(K_{m^{-\infty}\alpha}/K)$ with their respective normalised Haar measures.

\begin{prop}\label{alien}
If $L/K$ is any Galois extension that is linearly disjoint from $K_{m^{-\infty}\alpha}$ over $K$, then we have $\Dens_L(\alpha)=\Dens_K(\alpha)$.
\end{prop}
\begin{proof}
The generalisation of \cite[Proposition 20]{LombardoPerucca} to the composite case is  straightforward.
\end{proof}

\subsection{Counting elements in the image of the arboreal representation}

\begin{defi}\label{WuwuM}
For $M\in \imGal(m^n)$ we define
\begin{equation}\label{WuM}
{\set}_{m^n}(M):=\{ t\in A[m^{n}] \mid (t,M)\in \Gal({K_{m^{-n}\alpha}/K)}\}
\end{equation}
and
\begin{equation}\label{wuM}
\function_{m^n}(M):=\frac{\# \bigl(\im(M-I) \cap {\set}_{m^n}(M)\bigr)}{\# \im(M-I)} \in \mathbb Q\,.
\end{equation}
\end{defi}

For every prime divisor $\ell$ of $m$ and every $n\geqslant 1$, we
consider the Galois group of the compositum $K_{\ell^{-n}\alpha}K_{m^{-1}}$ over~$K$ and the inclusion
$$
\iota_{\ell^n}\colon
\Gal(K_{\ell^{-n}\alpha}K_{m^{-1}}/K)\hookrightarrow  \bigl(A[\ell^n] \rtimes \imGal(\ell^n)\bigr) \times \imGal(m).
$$

\begin{defi}\label{WuwuV}
For all $x\in \imGal(m)$ and $V\in \imGal(\ell^n)$, we define
\begin{equation}\label{WuV}
{\set}_{x,\ell^n}(V):=\{ \tau\in A[\ell^{n}] \mid (\tau, V, x)\in \im\iota_{\ell^n}\}
\end{equation}
and
\begin{equation}\label{wuV}
\function_{x,\ell^n}(V):=\frac{\# \bigl(\im(V-I) \cap {\set}_{x,\ell^n}(V)\bigr)}{\# \im(V-I)}\in\mathbb Z[1/\ell]\,.
\end{equation}
\end{defi}

We denote by $\pi_{*}$ the projection onto $\imGal(*)$.

\begin{prop}
If $M\in \imGal(m^n)$ is such that $\pi_{m}M=x$, then we have
\begin{equation}\label{bing}
\function_{m^n}(M)=\prod_\ell \function_{x, \ell^n}(\pi_{\ell^n} M)\,.
\end{equation}
\end{prop}
\begin{proof}
Decomposing $A[m^n]=\bigoplus_\ell A[\ell^n]$, we can write $t=\sum_\ell t_\ell$. We have $\im(M-I)=\bigoplus_\ell \im(\pi_{\ell^n} M-I)$, and since the extensions $K_{\ell^{-n}\alpha}K_{m^{-1}}/K_{m^{-1}}$ have pairwise coprime degrees and hence are linearly disjoint, we have $t\in {\set}_{m^n}(M)$ if and only if for every $\ell$ we have  $t_\ell\in {\set}_{x,\ell^n}(\pi_{\ell^n} M)$.
This implies the claim.
\end{proof}

\begin{lem}\label{bingo}
For all $x\in\imGal(m)$ and $V\in\imGal(\ell^\infty)$, the value $\function_{x,\ell^n}(V)$ is constant for $n$ sufficiently large.
\end{lem}

\begin{proof}
This is proved as in \cite[Lemma 23]{LombardoPerucca}.
\end{proof}

By Lemma~\ref{bingo}, we can define
\begin{equation}\label{function-ell-inf}
\function_{x,\ell^\infty}(V)=\lim_{n\to\infty}\function_{x,\ell^n}(V)\in\mathbb Z[1/\ell]\,.
\end{equation}
From \eqref{bing} we deduce that for all $M\in\imGal(m^\infty)$, the value $\function_{m^n}(M)$ is also constant for $n$ sufficiently large, so we can analogously define
\begin{equation}\label{function-m-inf}
\function_{m^\infty}(M)=\lim_{n\to\infty}\function_{m^n}(M) \in \mathbb Q\,.
\end{equation}

\begin{prop}
If $M\in \imGal(m^\infty)$ is such that $\pi_{m}M=x$, then we have
\begin{equation}\label{binginfty}
\function_{m^\infty}(M)=\prod_\ell \function_{x, \ell^\infty}(\pi_{\ell^\infty} M)\,.
\end{equation}
\end{prop}
\begin{proof}
Taking the limit as $n\to\infty$ in \eqref{bing} yields \eqref{binginfty}.
\end{proof}

\subsection{Partitioning the image of the $m$-adic representation}\label{sec:partition}
We view elements of $\imGal({m^\infty})$ as automorphisms of $A[m^\infty]=\bigcup_{n\geqslant1}A[m^n]$. We then classify elements $M\in \imGal({m^\infty})$ according to the group structure of $\ker(M-I)$ and according to the projection $\pi_m(M)\in\imGal({m})$.  Note that if $\ker(M-I)$ is finite, then it is a product over the primes $\ell\mid m$ of finite abelian $\ell$-groups that have at most $b_A$ cyclic components.

For every group $\fine$ of the form $\prod_\ell \fine_\ell$, where $\fine_\ell$ is a finite abelian $\ell$-group with at most $b_A$ cyclic components, we define the sets
\begin{equation}\label{MGH}
\mathcal M_{\fine}:=\{M\in \imGal(m^\infty): \, \ker \left (M-I : A[m^\infty] \to A[m^\infty]\right) \cong  \fine \}
\end{equation}
and
\begin{equation}\label{MGHx}
\mathcal M_{x,\fine}:=\{M\in \imGal({m^\infty}): \, \ker \left (M-I : A[m^\infty] \to A[m^\infty]\right) \cong  \fine ,\, \pi_{m}(M)=x \}\,.
\end{equation}
We denote by $\mathcal{M}_{\fine}(*)$ and $\mathcal{M}_{x,\fine}(*)$, respectively, the images of these sets under the reduction map $\imGal({m^\infty}) \to \imGal(*)$.
We also write
\begin{equation}\label{MathcalM}
\mathcal M := \bigcup_{\fine} \mathcal{M}_{\fine} = \bigcup_{x,\fine} \mathcal{M}_{x,\fine}\,,
\end{equation}
the union being taken over all $x\in \imGal(m)$ and over all groups $\fine=\prod_\ell \fine_\ell$ as above, up to isomorphism.

\begin{prop}\label{partition} The following  holds:
\begin{enumerate}
\item The sets $\mathcal{M}_{x,\fine}$ are measurable in $\imGal({m^\infty})$.
\item If $n>v_\ell(\exp \fine)$ for all $\ell\mid m$, then we have
$$
\mu_{\imGal(m^\infty)} (\mathcal{M}_{x,\fine})=\mu_{\imGal(m^n)}(\mathcal{M}_{x,\fine}(m^n)).
$$
\item We have $\mu_{\imGal(m^\infty)} (\mathcal{M}_{x,\fine})=0$ if and only if $\mathcal{M}_{x,\fine}=\emptyset$.
\item The set $\mathcal M$ of \eqref{MathcalM} is measurable in $\imGal({m^\infty})$.
\item If $(A/K,m)$ satisfies eventual maximal growth of the torsion fields, then we have
$$
\mu_{\imGal(m^\infty)}(\mathcal{M})=1.
$$
\end{enumerate}
\end{prop}

\begin{proof}
This is proved as in \cite[Lemma 21]{LombardoPerucca}.
\end{proof}

\section{The density as an integral}

By Theorem \ref{thm:interpretation}, computing $\Dens_m(\alpha)$ comes down to computing the Haar measure of $S_\alpha$ in $\Gal(K_{m^{-\infty} \alpha}/K)$.
The generalisation of \cite[Remark 17]{LombardoPerucca} to the composite case gives
$$
S_\alpha=\{(t,M)\in \Gal(K_{m^{-\infty} \alpha}/K)\;:\; M\in \imGal({m^\infty})\text{ and }  t\in \im (M-I)\}\,.
$$
In view of \eqref{MathcalM}, we consider the sets
\begin{equation}
S_{x,\fine}:= \{(t,M)\in \Gal(K_{m^{-\infty} \alpha}/K)\, :\; M\in \mathcal M_{x,\fine} \text{ and } t\in \im (M-I)\}\,.
\end{equation}

By Proposition~\ref{partition}, the set $S_\alpha$ is the disjoint union of the sets $S_{x,\fine}$ up to a set of measure $0$.
To see that the Haar measure of $S_{x,\fine}$ is well defined and to compute it, we define for every $n\geqslant 1$ the set
\begin{equation}
S_{x,\fine,m^n}=\{(t, M)\in \Gal(K_{m^{-n}\alpha}/K)\, :  M \in \mathcal{M}_{x,\fine}(m^n) \text{ and } t\in \im(M-I)\}\,.
\end{equation}

\begin{prop}
Suppose $n>n_0$ and $n>\max_\ell \{v_\ell(\exp \fine)\}$ for every $\ell$, where $n_0$ is as in Definition \ref{defi-conditions}. Then the set $S_{x,\fine,m^n}$ is the image of $S_{x,\fine}$ under the projection to $\Gal(K_{m^{-n}\alpha}/K)$.
\end{prop}
\begin{proof}
The set $S_{x,\fine,m^n}$ clearly contains the reduction modulo $m^n$ of $S_{x,\fine}$.
To prove the other inclusion, consider $(t_{m^n}, M_{m^n})\in S_{x,\fine,m^n}$ and a lift
$(t, M)\in \Gal(K_{m^{-\infty} \alpha}/K)$. Since $n$ is sufficiently large with respect to $F$, we have $\ker(M-I)\cong F$. Clearly $M_{m^n}$ and $M$ have the same projection $x\in \imGal(m)$. To conclude, it suffices to ensure $t\in \im (M-I)$.
Take $\tau_{m^n} \in A[m^n]$ satisfying $(M_{m^n}-I)(\tau_{m^n})=t_{m^n}$, and some lift $\tau$ of $\tau_{m^n}$ to $ T_m(A)$: we may replace $t$ by $(M-I)\tau$ because the difference is in $m^n T_m(A)$ and since $n>n_0$ we know that $\Gal(K_{m^{-\infty}\alpha}/K)$ contains $m^n T_m(A)\times \{I\}$.
\end{proof}

\begin{thm}
We have
\begin{equation}\label{Sxfine}
\mu(S_{x,\fine})= \frac{\failure_m}{\#\fine} \int_{\mathcal{M}_{x,\fine}} \function_{m^\infty}(M) \,\,  d\mu_{\imGal({m^\infty})} (M)
\end{equation}
where $\failure_m$ is the constant of \eqref{defi-failure} and $\function_{m^\infty}$ is as in \eqref{function-m-inf}.
\end{thm}

\begin{proof}
Choose $n$ large enough so that $n>n_0$ and $n>\max_\ell \{v_\ell(\exp \fine)\}$ for every $\ell$, where $n_0$ is as in Definition \ref{defi-conditions}.
By \eqref{WuM}, we can write
\[
\#S_{x,\fine, m^n}  = \sum_{M\in \mathcal M_{x,\fine}(m^n)}\# \bigl(\im(M-I) \cap \set_{m^n}(M)\bigr)\,.
\]
By \eqref{wuM}, we can express the summand as
$$
{\#\im(M-I)\cdot \function_{m^n}(M) } = \frac{\function_{m^n}(M)\cdot {m^{bn}}}{\#\fine}\,,
$$
so from \eqref{defi-failure} we deduce
$$
\frac{\#S_{x,\fine, m^n}}{\# \Gal({K_{m^{-n}\alpha}/K)}}=   \frac{1}{\#\imGal({m^n})} \sum_{M\in \mathcal M_{x,\fine}(m^n)}   \frac{\failure_m}{\#\fine}\cdot {\function_{m^n}(M)}\,.
$$
By \eqref{refomulation-Defi} the left-hand side is a non-increasing function of $n$, and therefore it admits a limit for $n\rightarrow \infty$, which is  $\mu(S_{x,\fine})$.
The right-hand side is an integral over $\mathcal M_{x,\fine}(m^n)$ with respect to the normalized counting measure of ${\imGal}({m^n})$, and the matrices in $\mathcal M_{x,\fine}$ are exactly the matrices in $\imGal({m^\infty})$ whose reduction modulo $m^n$ lies in $\mathcal M_{x,\fine}(m^n)$. Taking the limit in $n$ we thus find the formula in the statement.
\end{proof}

\begin{thm}\label{madicintegral}
We have
\begin{equation}\label{density-m111}
\Dens_m(\alpha) = \failure_m \int_{\imGal({m^\infty}) } \frac{\function_{m^\infty}(M)}{\# \ker (M-I)} \,\,  d\mu_{\imGal({m^\infty}) }(M)
\end{equation}
where the function $\function_{m^\infty}$ is as in \eqref{function-m-inf}, and the constant $\failure_m$ is as in \eqref{defi-failure}.
\end{thm}
Recalling the sets $\mathcal{M}_{\fine}$ from \eqref{MGH}, and that by Proposition~\ref{partition} their union has measure $1$ in $\imGal(m^\infty)$, we can rewrite
\eqref{density-m111} as
\begin{equation}\label{density-m1}
\Dens_m(\alpha) = \failure_m \sum_\fine \frac{1}{\#\fine} \int_{\mathcal{M}_{\fine}} \function_{m^\infty}(M) \,\,  d\mu_{\imGal({m^\infty})}(M)
\end{equation}
where $F$ varies over the products over the primes $\ell\mid m$ of finite abelian $\ell$-groups with at most $b_A$ cyclic components.

\begin{proof}
We have just shown that we may prove \eqref{density-m1} instead. Notice  that we have $\mathcal{M}_{\fine}=\cup_{x} \mathcal{M}_{x,\fine}$.
By Theorem \ref{thm:interpretation} we may write $\Dens_m(\alpha)= \mu(S_{\alpha})= \sum_{x,F}\mu(S_{x,\fine})$ and then it suffices to apply \eqref{Sxfine}.
\end{proof}

\begin{cor}[{\cite[Theorem 1 and Remark 25]{LombardoPerucca}}]
In the special case $m=\ell$ we have
\begin{equation}\label{density-ell1}
\Dens_\ell(\alpha) = \failure_\ell \sum_\fine \frac{1}{\#\fine} \int_{\mathcal{M}_{\fine}} \function_{\ell^\infty}(M) \,\, d\mu_{\imGal({\ell^\infty})}(M)
\end{equation}
where $F$ varies among the finite abelian $\ell$-groups with at most $b_A$ cyclic components.
We can also write
\begin{equation}\label{density-ell2}
\Dens_\ell(\alpha) = \failure_\ell \int_{\imGal(\ell^\infty)} \frac{\function_{\ell^\infty}(M)}{\# \ker (M-I)}\,\, d\mu_{\imGal({\ell^\infty})}(M)\,.
\end{equation}
\end{cor}
\begin{proof}
It suffices to specialize the notation of \eqref{density-m111} and  \eqref{density-m1}  to $m=\ell$.
\end{proof}

Notice that we have $\#\ker (M-I)=\ell^{v_\ell(\det(M-I))}$ for every $M\in \imGal(\ell^\infty)$.

\begin{cor}\label{corproduct}
If the fields $K_{\ell^{-\infty}\alpha}$ are linearly disjoint over $K$ (this happens for example if $K_{m^{-1}}=K$, or more generally if the degree $[K_{\ell^{-1}}:K]$ is a power of $\ell$ for each $\ell$), then we have
\begin{equation}
\Dens_m(\alpha) = \prod_{\ell} \Dens_\ell(\alpha).
\end{equation}
\end{cor}
\begin{proof}
Let $\ell$ vary among the prime divisors of $m$. Note that we have $\failure_m=\prod_\ell \failure_\ell$.
By assumption, we also know $\imGal({m^\infty})=\prod_\ell \imGal({\ell^\infty})$ and $\function_{m^\infty}(M)=\prod_\ell \function_{\ell^\infty}(\pi_{\ell^\infty} M)$. We may then apply Proposition \ref{prop:integral} and find that \eqref{density-m111} is the product of the expressions \eqref{density-ell2}.
\end{proof}

\begin{thm}\label{thmmany}
Write $\imGal({m^\infty})\subseteq \prod_\ell \imGal({\ell^\infty})$ as finite disjoint union of sets $H_i=\prod_\ell H_{i, \ell}$ where $H_{i, \ell}\in \imGal({\ell^\infty})$ are translates of closed subgroups of finite index.
Suppose that for every $\ell$ and for every $n\geqslant 1$ we know $[K_{m^{-1}} K_{\ell^{-n}\alpha}: K_{m^{-1}} K_{\ell^{-n}}]=[K_{\ell^{-n}\alpha}:K_{\ell^{-n}}]$. Then we have
\begin{equation}
\Dens_m(\alpha) = \sum_i \prod_{\ell} \failure_\ell \int_{H_{i,\ell}} \frac{\function_{\ell^\infty}(M)}{\# \ker (M-I)} d\mu_{\imGal({\ell^\infty})}(M).
\end{equation}
\end{thm}
\begin{proof}
We can apply Proposition \ref{prop:integral} because there is a product decomposition formula for the function $\function_{m^\infty}$. Namely the assumption on the Kummer extensions  implies that for every $n\geqslant 1$ we have $\function_{m^n}(M)=\prod_\ell \function_{\ell^n}(\pi_{\ell^n} M)$ and hence we find
$\function_{m^\infty}(M)=\prod_\ell \function_{\ell^\infty}(\pi_{\ell^\infty}(M))$.
\end{proof}

\begin{thm}\label{thm:densityproduct}
Denote by $H(x)= \prod_\ell H_\ell(x)$ the set of matrices in $\imGal({m^\infty})\subseteq \prod_\ell \imGal({\ell^\infty})$ mapping to $x$ in $\imGal({m})$.
We then have
\begin{equation}\label{bigformuladensity}
\Dens_m(\alpha) = \failure_m \cdot \sum_{x\in \imGal(m)} \prod_{\ell}
\int_{H_\ell(x)} \frac{\function_{x,\ell^\infty}(M)}{\# \ker (M-I)}
\,\, d\mu_{H_\ell(x)}M
\end{equation}
where $\function_{x,\ell^\infty}$ is as in \eqref{function-ell-inf}.
\end{thm}
\begin{proof}
Write $S_{x}=\bigcup_\fine S_{x,\fine}$ and recall from Proposition~\ref{partition} that the set of matrices $M$ for which $\ker(M-I)$ is infinite has measure zero in $\imGal({m^\infty})$. By \eqref{Sxfine}
we have
\begin{equation}\label{musg}
\mu(S_{x})=\sum_\fine \mu(S_{x, \fine})= {\failure_m} \int_{H(x)}
\frac{\function_{m^\infty}(M)}{\# \ker (M-I)}
 d\mu_{\imGal({m^\infty})}(M)\,.
\end{equation}
By \eqref{binginfty} the assertion follows from Theorem \ref{thm:decomposition}. \end{proof}

\begin{cor}\label{wunder}
The density $\Dens_m(\alpha)$ is a rational number. Moreover, fix $g \geqslant 1$. There exists a polynomial $p_g(t) \in \mathbb{Z}[t]$ with the following property: whenever $K$ is a number field and $A$ is the product of an abelian variety and a torus with $\dim(A)=g$, then for all $\alpha \in A(K)$ we have
\[
\Dens_m(\alpha) \cdot \prod_\ell p_g(\ell) \in \mathbb{Z}[1/m]\,,
\]
where $\ell$ varies over the prime divisors of $m$.
\end{cor}
\begin{proof}
Recall that $\failure_m$ is an integer. In view of Lemma~\ref{bingo}, we can consider each $\ell$-adic integral in \eqref{bigformuladensity} and proceed as in the proof of \cite[Theorem 33]{LombardoPerucca}.
\end{proof}

\section{Serre curves}\label{serrecurves}

\subsection{Definition of Serre curves}

Let $E$ be an elliptic curve over a number field $K$.  We choose a Weierstrass equation for~$E$ of the form
\begin{equation}\label{Weierstrass}
E\colon y^2 = (x - x_1)(x - x_2)(x - x_3)\,,
\end{equation}
where $x_1, x_2, x_3 \in K(E[2])$ are the $x$-coordinates of the points of order~$2$.
The discriminant of the right-hand side of~\eqref{Weierstrass} is $\Delta=\sqrt{\Delta}^2$, where
$$
\sqrt{\Delta}=(x_1-x_2)(x_2-x_3)(x_3-x_1)\,.
$$
We thus have $K(\sqrt{\Delta})\subseteq K(E[2])$, and we define a character
$$
\begin{aligned}
\psi_E\colon \Gal(K(E[2])/K)&\longrightarrow\{\pm1\}\\
\sigma &\longmapsto \sigma(\sqrt{\Delta})/\sqrt{\Delta}\,.
\end{aligned}
$$
For any choice of basis of the $2$-torsion of~$E$, we have the 2-torsion representation
$$
\rho_{E,2}\colon\Gal(K(E[2])/K)\longrightarrow\GL_2(\Z/2\Z)\,.
$$
Let $\psi$ be the unique non-trivial character $\GL_2(\Z/2\Z)\to\{\pm1\}$; this corresponds to the sign character under any isomorphism of $\GL_2(\Z/2\Z)$ with $S_3$.  The character $\psi_E$ factors as
$$
\psi_E = \psi \circ \rho_{E,2}\,.
$$

From now on, we take $K=\mathbb Q$. All number fields that we will consider will be subfields of a fixed algebraic closure $\overline{\mathbb Q}$ of $\mathbb Q$.

Let $d$ be an element of $\mathbb Q^\times$.
Let $m_d$ be the conductor of $\mathbb Q(\sqrt{d})$; this is the smallest positive integer such that $\sqrt{d}$ lies in the cyclotomic field $\mathbb Q(\zeta_{m_d})$.
Let $d_\sqf$ be the square-free part of~$d$.  We have
$$
m_d=\left\{\begin{array}{ll}
|d_\sqf| & \text{if } d_\sqf\equiv 1 \bmod 4,\\
4|d_\sqf| & \text{otherwise}.\\
\end{array}\right.
$$
We define a character
$$
\begin{aligned}
\varepsilon_d\colon\Gal(\mathbb Q(\zeta_{m_d})/\mathbb Q)&\longrightarrow\{\pm1\}\\
\sigma&\longmapsto\sigma(\sqrt{d})/\sqrt{d}\,.
\end{aligned}
$$
If $\sigma$ is the automorphism of $\mathbb Q(\zeta_{m_d})$ defined by $\sigma(\zeta_{m_d})=\zeta_{m_d}^a$ with $a\in(\mathbb Z/m_d \mathbb Z)^\times$, then $\varepsilon_d(\sigma)$ equals the Jacobi symbol $\bigl(\frac{d_\sqf}{a}\bigr)$.
We view $\varepsilon_d$ as a character of $\GL_2(\mathbb Z/m_d \mathbb Z)$ by composing with the determinant.

For all $n\geqslant 1$, we have a canonical projection
$$
\pi_n: \GL_2(\Zhat)\rightarrow \GL_2(\mathbb Z/n\mathbb Z)\,.
$$
Fixing a $\Zhat$-basis for the projective limit of the torsion groups $E[n](\Qbar)$, we have a torsion representation
$$
\rho_E\colon\Gal(\Qbar/\mathbb Q)\rightarrow \GL_2(\Zhat)\,.
$$
The image of $\rho_E$ is contained in the subgroup
$$
H_\Delta = \{M\in\GL_2(\Zhat)\mid
\psi(\pi_2(M)) = \varepsilon_\Delta(\pi_{m_\Delta}(M))\}
$$
of index~$2$ in $\GL_2(\Zhat)$.
This expresses the fact that $\sqrt{\Delta}$ is contained in both $\mathbb Q(E[2])$ and $\mathbb Q(E[m_\Delta])$.
An elliptic curve is said to be a \emph{Serre curve} if the image of $\rho_E$ is \emph{equal} to~$H_\Delta$.
As proven by N. Jones \cite{NathanJones}, almost all elliptic curves over $\mathbb Q$ are Serre curves.

\subsection{Counting matrices}

Let $\ell$ be a prime number. For all integers $a,b\geqslant 0$, we write $\mathcal M_\ell(a,b)$ for the set of matrices $M\in\GL_2(\Z_\ell)$ such that the kernel of $M-I$ as an endomorphism of $(\Q_\ell/\Z_\ell)^2$ is isomorphic to $\Z/\ell^{a}\mathbb Z\times \Z/\ell^{a+b}\mathbb Z$.

If $\mathcal N$ is a subset of $\mathcal M_\ell(a,b)$ that is the preimage in $\mathcal M_\ell(a,b)$ of its reduction modulo $\ell^n$, then we have
\begin{equation}\label{proportionality}
\frac{\mu_{\GL_2(\Z_\ell)}(\mathcal N)}{\mu_{\GL_2(\Z_\ell)}(\mathcal M_\ell(a,b))}=\frac{\mu_{\GL_2(\Z/\ell^n\Z)}(\mathcal N \bmod\ell^n)}{\mu_{\GL_2(\Z/\ell^n\Z)}(\mathcal M_\ell(a,b) \bmod\ell^n)}\,.
\end{equation}
by  \cite[Theorem 27]{LombardoPeruccaNY} (where the number of lifts is independent of the matrix).

\begin{prop}\label{prop-uniform-lift}
If $\mathcal N$ is a subset of $\mathcal M_\ell(a,b)$ that is the preimage in $\mathcal M_\ell(a,b)$ of its reduction modulo $\ell$, then we have
\[
\mu_{\GL_2(\Z_\ell)}(\mathcal N) = \mu_{\GL_2(\Z/\ell\Z)}(\mathcal N\bmod\ell) \cdot \left\{ \begin{array}{ll}
1 & \text{if $a=b=0$}\\
\ell^{-b} (\ell-1) & \text{if $a=0$, $b\geqslant 1$} \\
\ell^{-4a}\cdot  \ell (\ell-1)^2(\ell+1) & \text{if $a\geqslant 1, b=0$} \\
\ell^{-4a-b}\cdot (\ell-1)^2(\ell+1)^2 & \text{if $a\geqslant 1, b\geqslant 1$}\,.
\end{array}\right.
\]
\end{prop}
\begin{proof}
We are working with $\GL_2(\Z_\ell)$, so we can apply \cite[Proposition 33]{LombardoPeruccaNY} (see also \cite[Definition 19]{LombardoPeruccaNY}). This gives the assertion for the set $\mathcal M_\ell(a,b)$; we can conclude because of \eqref{proportionality}.
\end{proof}

We now collect some results in the case $\ell=2$.
From \cite[Theorem 2]{LombardoPeruccaNY} we know
\begin{equation}
\mu_{\GL_2(\Z_2)}(\mathcal M_2(a,b)) = \left\{ \begin{array}{ll}
{1}/{3} & \text{if $a=b=0$}\\
{1}/{2}\cdot 2^{-b}  & \text{if $a=0$, $b\geqslant 1$} \\
2^{-4a}  & \text{if $a\geqslant 1, b=0$} \\
{3}/{2}\cdot 2^{-4a-b} & \text{if $a\geqslant 1, b\geqslant 1$}\,.
\end{array}\right.
\end{equation}

We consider the action of $\GL_2(\Z/2^3\Z)$ on $\Q(\zeta_{2^3})$ defined by $M\zeta_{2^3}=\zeta_{2^3}^{\det M}$. The matrices $M\in \GL_2(\mathbb Z/2^3\mathbb Z)$ that fix $\sqrt{-1}$ are those with $\det(M)=1,5$.
The ones that fix $\sqrt{2}$ are those with $\det(M)=1,7$.
The ones that fix $\sqrt{-2}$ are those with $\det(M)=1,3$.

For $a,b\in\{0,1,2,3\}$ and $z\in\{-1,2,-2\}$, we write $\mathcal N_2(a,b,z)$ for the set of matrices in $\mathcal M_2(a,b)$ that fix $\sqrt{z}$.

\begin{lem}\label{classificatioN}
We have
\begin{equation}
\frac{\mu_{\GL_2(\Z_2)}(\mathcal N_2(a,b;{-1}))}
{\mu_{\GL_2(\Z_2)}(\mathcal M_2(a,b))}=
\left\{
\begin{array}{ll}
1/2 & \text{for $a=0$, $b\geqslant 0$}\\
2/3 & \text{for $a=1$, $b= 0$}\\
1/3 & \text{for $a=1$, $b\geqslant 1$}\\
1 & \text{for $a\geqslant 2$, $b\geqslant 0$}\\
\end{array}\right.
\end{equation}
and
\begin{equation}
\frac{\mu_{\GL_2(\Z_2)}(\mathcal N_2(a,b; \pm {2}))}
{\mu_{\GL_2(\Z_2)}(\mathcal M_2(a,b))}=
\left\{
\begin{array}{ll}
1/2 & \text{for $a\leqslant 1$, $b\geqslant 0$}\\
2/3 & \text{for $a=2$, $b= 0$}\\
1/3 & \text{for $a=2$, $b\geqslant 1$}\\
1 & \text{for $a\geqslant 3$, $b\geqslant 0$\rlap.}\\
\end{array}\right.
\end{equation}
\end{lem}
\begin{proof}
For $a,b\in \{0,1,2,3\}$ and $d\in(\Z/2^3\Z)^\times$, let $h(a,b,d)$ be the number of matrices $M\in \GL_2(\mathbb Z/2^3\mathbb Z)$ such that $\det(M)=d$ and $\ker (M-I)\cong \mathbb Z/2^a\mathbb Z\times \mathbb Z/2^{a+b}\mathbb Z$. Using \cite{sagemath} one can easily count these matrices:
\begin{itemize}
\item $h(0,0,d)=128$, $h(0,1,d)=96$ and $h(0,2,d)=h(0,3,d)=48$ for all $d$;
\item $h(1,0,d)=32$ for $d=1,5$ and $h(1,0,d)=16$ for $d=3,7$;
\item for $b=1,2$ we have $h(1,b,d)=12$ for $d=1,5$ and $h(1,b,d)=24$ for $d=3,7$;
\item $h(2,0,1)=4$, $h(2,0,5)=2$ and $h(2,0,d)=0$ for $d=3,7$;
\item $h(2,1,1)=3$, $h(2,1,5)=6$ and $h(2,1,d)=0$ for $d=3,7$;
\item $h(3,0,1)=1$ (the identity matrix) and $h(3,0,d)=0$ for $d=3,5,7$.
\end{itemize}
This classification and \eqref{proportionality} lead to the measures in the statement.
\end{proof}

\begin{lem}\label{lemma:two}
For all $a,b\ge0$ and all $M\in\mathcal M_2(a,b)$, we have
$$
\psi(M) = \begin{cases}
-1& \text{if $a=0$ and }b\geqslant 1,\cr
1& \text{otherwise}.
\end{cases}
$$
\end{lem}

\begin{proof}
Consider matrices $M\in \GL_2(\mathbb Z/2\mathbb Z)$. The matrices
$$
M\in\left\{\begin{pmatrix} 1 & 0 \\ 0 & 1 \end{pmatrix}, \quad
\begin{pmatrix} 0 & 1 \\ 1 & 1 \end{pmatrix}, \quad
\begin{pmatrix} 1 & 1 \\ 1 & 0 \end{pmatrix}\right\}
$$
satisfy $\psi(M)=1$ and $\dim_{\F_2}\ker(M-I)\in\{0,2\}$.  The matrices
$$
M\in\left\{\begin{pmatrix} 0 & 1 \\ 1 & 0 \end{pmatrix}, \quad
\begin{pmatrix} 1 & 1 \\ 0 & 1 \end{pmatrix}, \quad
\begin{pmatrix} 1 & 0 \\ 1 & 1 \end{pmatrix}\right\}
$$
satisfy $\psi(M)=-1$ and $\dim_{\F_2}\ker(M-I)=1$. This implies the claim.
\end{proof}

Now let $\ell$ be an odd prime number, and suppose that $M\in\GL_2(\mathbb Z/\ell\mathbb Z)$ is not the identity. We compare the two conditions $\varepsilon_\ell (M)=1$ and $\ell \mid \det(M-I)$.
Note that the characteristic polynomials of $M$ and $M-I$ have roots differing only by a shift by $1$. Write $\chi (M)=x^2+\alpha x+\beta$ for the characteristic polynomial of $M$.
The condition $\varepsilon_\ell (M)=1$ is a condition on $\det(M)=\beta$.
More precisely, we want $\beta$ to be a square modulo $\ell$, and in this case we write $\beta=s^2$.
On the other hand we would like to know whether $\ell \mid \det(M-I)=\chi(1)$, which means
$1+\alpha+\beta=0$ (we are working modulo $\ell$). Thus $\alpha=-(1+s^2)$.
We then have
\begin{equation}\label{charpol}
\chi(M)=x^2-(1+s^2)x+s^2=(x-1)(x-s^2)=(x-1)(x-\beta)\,.
\end{equation}

\begin{lem}\label{lemma-classification}
Let $M$ vary in $\GL_2(\mathbb Z/\ell\mathbb Z)\setminus\{I\}$, where $\ell$ is an odd prime number.
\begin{enumerate}
\item There are $\frac{1}{2}(\ell+1)^2(\ell-2)$ matrices $M$ satisfying $\varepsilon_\ell (M)=1$ and $\ell \mid \det(M-I)$.
\item There are $\frac{1}{2}\ell (\ell^3 - 2 \ell^2 -\ell + 4)$ matrices $M$ satisfying $\varepsilon_\ell (M)=1$ and $\ell \nmid \det(M-I)$.
\item There are $\frac{1}{2} \ell(\ell^2-1)$ matrices $M$ satisfying $\varepsilon_\ell (M)=-1$ and $\ell \mid \det(M-I)$.
\item There are $\frac{1}{2} \ell(\ell^2-1)(\ell-2)$ matrices $M$ satisfying $\varepsilon_\ell (M)=-1$ and $\ell \nmid \det(M-I)$.
\end{enumerate}
\end{lem}
\begin{proof}
\begin{enumerate}
\item In \eqref{charpol}, if $\beta=s^2\neq 1$ (thus there are $\frac{\ell-1}{2}-1$ possibilities for $\beta$), then the corresponding matrices are diagonal and hence we only have to choose the two distinct eigenspaces; this gives
$(\ell+1)\ell$ matrices for every characteristic polynomial.
If $\beta=1$, then we have the identity (which we are excluding) and the $\ell^2-1$ matrices  conjugate to
$\begin{pmatrix} 1 & 1 \\ 0 & 1 \end{pmatrix}$.
Note that (1) can also be obtained from \cite[Table 1]{Sutherland16}.

\item There are $\frac{1}{2}\#\GL_2(\mathbb Z/\ell\mathbb Z)$ matrices satisfying $\varepsilon_\ell=1$, and we only need to subtract the identity and the matrices from (1).

\item There are $\ell^3-2\ell$ matrices in $\GL_2(\mathbb Z/\ell\mathbb Z)$ having $1$ as an eigenvalue (see for example \cite[Proof of Theorem 2]{LombardoPeruccaNY}), and we only need to subtract the identity and the matrices from (1).

\item There are $\frac{1}{2}\#\GL_2(\mathbb Z/\ell\mathbb Z)$ matrices satisfying $\varepsilon_\ell=-1$, and we only need to subtract the matrices from (3). Alternatively, there are $\#\GL_2(\mathbb Z/\ell\mathbb Z)-(\ell^3-2\ell)$ matrices that do not have $1$ as eigenvalue, and we only need to subtract the matrices from (2).
\end{enumerate}
\end{proof}

\subsection{Partitioning the image of the $m$-adic representation}\label{partitioning}

Let $E$ be a Serre curve over $\Q$.
Let $\Delta$ be the minimal discriminant of~$E$, and let $\Delta_\sqf$ be its square-free part.
We write $\Delta_\sqf=z u$, where $z \in \{1,-1,2,-2\}$ and where $u$ is an odd fundamental discriminant.  Then $|u|$ is the odd part of $m_\Delta$, and we have $\varepsilon_\Delta=\varepsilon_z \cdot \varepsilon_u$ as characters of $(\Z/m_\Delta\Z)^\times$.

Now let $m$ be a square-free positive integer.
If $m=2$, or if $m$ is odd, or if $u$ does not divide $m$, then we have
$$
\imGal({m}^\infty)=\prod_{\ell} \imGal({\ell}^\infty)\,.
$$
If $m\neq 2$ is even and $u$ divides $m$, then $\imGal(m^\infty)$ has index~$2$ in $\prod_{\ell} \imGal({\ell}^\infty)$. The defining condition for the image of the $m$-adic representation is then $\psi=\varepsilon_\Delta$, or equivalently
\begin{equation}\label{definingcondition}
\psi\cdot\varepsilon_z=\varepsilon_u\,.
\end{equation}
We may then partition $\imGal({m}^\infty)\subseteq \prod_{\ell\mid m} \imGal({\ell}^\infty)$ into two sets that are products, namely
$$
(\imGal({2}^\infty)\cap \{\psi\cdot \varepsilon_z=1\}) \times (\imGal(|u|^\infty)\cap \{\varepsilon_u=1\}) \times \imGal\Bigl(\Bigl|\frac{m}{2u}\Bigr|^\infty\Bigr)
$$
and
$$
(\imGal({2}^\infty)\cap \{\psi\cdot \varepsilon_z=-1\}) \times (\imGal(|u|^\infty)\cap \{\varepsilon_u=-1\}) \times \imGal\Bigl(\Bigl|\frac{m}{2u}\Bigr|^\infty\Bigr)\,.
$$
The set $\imGal(|u|^\infty)\cap \{\varepsilon_u= 1\}$ is the disjoint union of sets of the form $\prod_{\ell\mid u}(\imGal({\ell}^\infty)\cap \{\varepsilon_\ell=\pm 1\})$, choosing an even number of minus signs; for the set $\imGal(|u|^\infty)\cap \{\varepsilon_u= -1\}$ we have to choose an odd number of minus signs.
Since each $\ell\mid u$ is odd, the two sets $\imGal({\ell}^\infty)\cap \{\varepsilon_\ell=\pm 1\}$ can be investigated with the help of Lemma \ref{lemma-classification}.
Finally, the two sets $\imGal({2}^\infty)\cap \{\psi\cdot \varepsilon_z =\pm 1\}$ can be investigated using Lemmas \ref{classificatioN} and~\ref{lemma:two}.

\section{Examples}

\subsection{Example (non-surjective mod 3 representation)}\label{sub-exa1}
Consider the non-CM elliptic curve
$$
E: y^2 + y = x^3 + 6x + 27
$$
of discriminant $-3^{19}\cdot17$ and conductor $153=3^2\cdot17$ over $\mathbb{Q}$ \cite[\href{http://www.lmfdb.org/EllipticCurve/Q/153/b/2}{label 153.b2}]{lmfdb}.  The group $E(\Q)$ is infinite cyclic and is generated by the point
$$
\alpha=(5,13).
$$

The image of the $3$-adic representation is the inverse image of its reduction modulo $3$, the image of the mod $3$ representation is isomorphic to the symmetric group of order $6$, and the $3$-adic Kummer map is surjective \cite[Example 6.4]{LombardoPerucca}. The image of the mod $3$ representation has a unique subgroup of index~$2$, so the field $\Q(E[3])$ contains as its only quadratic subextension the cyclotomic field $\mathbb Q(\sqrt{-3})$.

The image of the $2$-adic representation is $\GL_2(\Z_2)$; see \cite{lmfdb}. By \cite[Theorem 5.2]{JonesRouse} the $2$-adic Kummer map is surjective: the assumptions of that result are satisfied because the prime $p=941$ splits completely in $E[4]$ but the point $(\alpha \bmod p)$ is not $2$-divisible over $\mathbb F_p$. Since the image of the mod~$2$ representation has a unique subgroup of index~$2$, the field $\Q(E[2])$ contains as its only quadratic subextension the field $\mathbb Q(\sqrt{-51})$ (the square-free part of the discriminant of $E$ is $-51$).

We have $\mathbb Q(E[2])\cap \mathbb Q(E[9]) = \Q$ because the residual degree modulo $22699$ of the extension $\mathbb Q(E[2], E[9])/\mathbb Q(E[9])$ is divisible by $3$ and the degree of this extension is even because $\Q(\sqrt{-51})$ is not contained in $\mathbb Q(E[3])$.
We deduce $\mathbb Q(E[2])\cap \mathbb Q(E[3^\infty]) = \mathbb Q$ by applying \cite[Theorem 12 (i)]{LombardoPerucca} (where $K=\Q(E[2])$).

Moreover, we have $\mathbb Q(E[3])\cap \mathbb Q(E[4]) = \Q$ because $\mathbb Q(\sqrt{-3})$ is not contained in $\mathbb Q(E[4])$: the prime $941$ is not congruent to $1$ modulo $3$ and splits completely in $\mathbb Q(E[4])$. By \cite[Theorem 12 (i)]{LombardoPerucca} we conclude that $\mathbb Q(E[3])\cap \mathbb Q(E[2^\infty]) = \mathbb Q$.

The $2$-adic Kummer extensions of $\alpha$ have maximal degree also over $\Q(E[3])$, in view of the maximality  of the $2$-Kummer extension, because the prime $4349$ splits completely in $\Q(2^{-2} \alpha)$ but not in $\Q(\sqrt{-3})$; see \cite[Theorem 12 (ii)]{LombardoPerucca} (where $K=\Q(\sqrt{-3})$).

The $3$-adic Kummer extensions of $\alpha$ have maximal degree also over $\Q(E[2])$ because the prime $217981$ splits completely in $\Q(3^{-2} \alpha)$ but $3$ divides the residual degree of $\Q(E[2])$; see \cite[Theorem 12 (ii)]{LombardoPerucca} (where $K=\Q(E[2])$).

We thus have $\imGal({6}^\infty)=\imGal({2}^\infty)\times \imGal({3}^\infty)$, the $2^\infty$ Kummer extensions are independent from $\Q(E[3])$, and the $3^\infty$ Kummer extensions are independent from $\Q(E[2])$. We are thus in the situation that the fields $\Q(2^{-\infty} \alpha)$ and $\Q(3^{-\infty} \alpha)$ are linearly disjoint over $\Q$. We deduce that the equality
$$\Dens_6(\alpha) = \Dens_2(\alpha) \cdot \Dens_3(\alpha)$$
holds for $\alpha$ and for its multiples.
The $2$-densities can be evaluated by \cite[Theorem 32]{LombardoPerucca}, for the $3$-densities see \cite[Example 6.4]{LombardoPerucca}. By testing the primes up to $10^6$, we have computed an approximation to $\Dens_6(\alpha)$ using \cite{sagemath}.

\begin{center}
\begin{tabular}{l|ll|r|l}
Point & $\Dens_2$ & $\Dens_3$ & $\Dens_6\hfil$ & primes $<10^6$\\
\hline
$\phantom1\alpha=(5, 13)$ & $11/21$ & $23/104$ & $253/2184=11.584\ldots\%$ &$11.624\%$\\
$2\alpha=(-1 , 4)$ & $16/21$ & $23/104$ & $46/273=16.849\ldots\%$ & $16.885\%$\\
$3\alpha=(-7/4,  -31/8)$ & $11/21$ & $77/104$ & $121/312=38.782\ldots\%$& $38.730\%$\\
$6\alpha=(137/16, 1669/64)$ & $16/21$ & $77/104$ & $22/39=56.410\ldots\%$& $56.373\%$\\
$4\alpha=(3, -9)$ & $37/42$ & $23/104$ & $851/4368=19.482 \ldots\%$ & $19.479\%$\\
$9\alpha=(\frac{19649}{12100}, -\frac{9216643}{1331000})$ & $11/21$ & $95/104$ & $1045/2184=47.847 \ldots\%$ & $47.791\%$\\
\end{tabular}
\end{center}

\subsection{The Serre curve $y^2+y=x^3+x^2$}\label{serre}

The elliptic curve
$$
E: y^2+y=x^3+x^2
$$
of discriminant $-43$ and conductor 43 over $\Q$ \cite[\href{http://www.lmfdb.org/EllipticCurve/Q/43/a/1}{label 43.a1}]{lmfdb} is a Serre curve \cite[Example 5.5.7]{Serre}. The group $E(\Q)$ is infinite cyclic and is generated by the point
$$
\alpha=(0,0).
$$
A computation shows that $\alpha$ is not divisible by~$2$ over the 4-torsion field of~$E$.  Therefore, by \cite[Theorem~5.2]{JonesRouse}, for every prime number $\ell$ and for every $n\geqslant 1$ the degree of the $\ell^n$-Kummer extension is maximal, i.e.
$$
[\Q_{\ell^{-n}\alpha}:\Q_{\ell^{-n}}]=\ell^{2n}.
$$

The $43$-adic Kummer extensions have maximal degree also over $\mathbb Q(E[2])$, i.e.
$$
[\Q_{43^{-n}\alpha}(E[2]):\Q_{43^{-n}}(E[2])]=43^{2n},
$$
because the degree $[\Q(E[2]):\Q]=6$ is coprime to $43$.

The extensions $\mathbb Q({2^{-1}\alpha})$ and $\mathbb Q(E[2\cdot 43])$ are linearly disjoint over $\mathbb Q(E[2])$, as can be seen by investigating the residual degree for the reduction modulo the prime $29327$, which splits completely in $\mathbb Q(E[2])$. Indeed, the residual degree of the extension $\mathbb Q({2^{-1}\alpha})$ equals $4$ while the residual degree of the extension $\mathbb Q(E[2\cdot 43])$ is odd because the prime is congruent to $1$ modulo $43$, and there are points of order $43$ in the reductions (the subgroup of the upper unitriangular matrices in $\GL_2(\mathbb Z/43\mathbb Z)$ has order $43$).

The $2$-adic Kummer extensions have maximal degree also over $\Q(E[43])$, i.e.
$$
[\Q_{2^{-n}\alpha}(E[43]):\Q_{2^{-n}}(E[43])]=2^{2n}.
$$
To see this, we consider the intersection $L$ of $\Q_{2^{-n}\alpha}$ and $\Q(E[43])$.  This is a Galois extension of~$\Q$, and the group $G=\Gal(L/\Q)$ is a quotient of both $(\Z/2^n\Z)^2\rtimes\GL_2(\Z/2^n\Z)$ and $\GL_2(\Z/43\Z)$.  Because $\SL_2(\Z/43\Z)$ has no non-trivial quotient that can be embedded into a quotient of $(\Z/2^n\Z)^2\rtimes\GL_2(\Z/2^n\Z)$, the quotient map $\GL_2(\Z/43\Z)\to G$ factors as
$$
\GL_2(\Z/43\Z)\stackrel{\det}\longrightarrow(\Z/43\Z)^\times\longrightarrow G
$$
This implies that $L$ is a subfield of $\Q(\zeta_{43})$.  Furthermore, $L$ contains $\Q(\sqrt{-43})$.  Using basic Galois theory, we see that the maximal subfield of $\Q(\zeta_{43})$ that can be embedded into $\Q_{2^{-n}\alpha}$ is $\Q(\sqrt{-43})$, and we conclude that $L$ equals $\Q(\sqrt{-43})$.

It follows that for $m=2\cdot 43$ we have the maximal degree $[\Q_{m^{-n}\alpha}:\Q_{m^{-n}}]=m^{2n}$ and, more generally, that for every multiple $P$ of $\alpha$ we have $[\Q_{m^{-n}P}:\Q_{m^{-n}}]=[\Q_{2^{-n}P}:\Q_{2^{-n}}]\cdot [\Q_{43^{-n}P}:\Q_{43^{-n}}]$. We may then apply \cite[Example 26]{LombardoPerucca} and various results in this paper to compute the following exact densities, and we use \cite{sagemath} to numerically verify them for the primes up to $10^6$:
\begin{center}
\begin{tabular}{l|r|l}
Point & $\Dens_{2\cdot 43}\hfil$ & primes $<10^6$\\
\hline
$\phantom1\alpha=(0,0)$ & $526206455/1028489616=51.163\ldots\%$ & $51.136\%$\\
$2\alpha=(-1 , -1)$ & $42521603/57138312=74.418\ldots\%$ & $74.397\%$\\
$4\alpha=(2,3)$ & $1769960107/2056979232=86.046\ldots\%$ & $86.072\%$\\
\end{tabular}
\end{center}

We conclude by sketching the computations for the point $\alpha$. The $43$-adic representation is surjective and the $43$-Kummer extensions have maximal degree. By parts (3) and (4) of Lemma \ref{lemma-classification},
we find that $\frac{1}{2\cdot 42}$ (respectively, $\frac{41}{2\cdot 42}$) is the counting measure in $\GL_2(\mathbb Z/43 \mathbb Z)$ of the matrices such that $\varepsilon_{43}=-1$ and that are in $(\mathcal M_{43}(0,b) \bmod \ell)$ for some $b>0$ (respectively, for $b=0$).
By multiplying this quantity by $43^{-b}\cdot 42$ we obtain by Proposition \ref{prop-uniform-lift} that $\mu_{\GL_2(\Z_{43})}(\mathcal M_{43}(0,b))=\frac{1}{2}43^{-b}$ for $b>0$.
By \cite[Example 26]{LombardoPerucca} the contribution to $\Dens_{43}$ coming from the matrices in $\imGal(43^\infty)$ such that $\varepsilon_{43}=-1$ is then
$$
\Dens_{43}(\varepsilon_{43}=-1)=\frac{41}{2\cdot 42}+\sum_{b>0} \frac{1}{2}\cdot 43^{-2b}=\frac{1805}{2\cdot 42\cdot 44}\,.
$$
From \cite[Theorem 32]{LombardoPerucca} we know that $\Dens_{43}(\alpha)=143510179/146927088$, and hence the contribution to $\Dens_{43}(\alpha)$ coming from the matrices in $\imGal(43^\infty)$ such that $\varepsilon_{43}=+1$ equals
$$
\Dens_{43}(\varepsilon_{43}=1)=\frac{3261637}{6678504}.
$$

Now we work with the $2$-adic representation, which is surjective, and restrict to counting the contribution to $\Dens_{2}(\alpha)$ coming from the matrices satisfying $\psi=-1$. In view of Lemma \ref{lemma:two} and Proposition \ref{prop-uniform-lift}, we find $\mu_{\GL_2(\Z_2)}(\mathcal M_2(0,b))=1/2 \cdot 2^{-b}$ for $b>0$.
By \cite[Example 26]{LombardoPerucca} the contribution to $\Dens_{2}(\alpha)$ coming from the matrices in $\imGal(2^\infty)$ such that $\psi=-1$ is therefore
\begin{equation}\label{conto}
\Dens_2(\psi=-1)=\sum_{b>0} {1}/{2}\cdot 2^{-2b} = 1/6.
\end{equation}
From \cite[Theorem 32]{LombardoPerucca} we know that $\Dens_{2}(\alpha)=11/21$, and hence the contribution to $\Dens_{2}$ coming from the matrices in $\imGal(2^\infty)$ such that $\psi=1$ is
$$
\Dens_2(\psi=1)=5/14.
$$
Finally by the partition in Section \ref{partitioning} we can compute the requested density as the combination of the above quantities:
$$
\Dens_{2\cdot 43}(\alpha) = 2\bigl( \Dens_{2}({\psi=1})\cdot\Dens_{43}({\varepsilon_{43}=1})+\Dens_{2}({\psi=-1})\cdot\Dens_{43}({\varepsilon_{43}=-1})\bigr).
$$
The factor $2$ accounts for the fact that the integration domain (equipped with its normalised measure) is a subgroup of index $2$ of $\GL_2(\mathbb Z_2)\times \GL_2(\mathbb Z_{43})$, while we compute $\Dens_{2}(\alpha)$ and $\Dens_{43}(\alpha)$ in $\GL_2(\mathbb Z_2)$ and $\GL_2(\mathbb Z_{43})$, respectively.

For the point $2\alpha$, by \cite[Example 26]{LombardoPerucca} we only need to scale \eqref{conto} by a factor $2$, giving $1/3$ and $3/7$ as the two contributions to $\Dens_{2}(2\alpha)$ by \cite[Theorem 32]{LombardoPerucca}. For the point $4\alpha$, we adapt \eqref{conto} as $2\cdot {1}/{2}\cdot 2^{-2}+\sum_{b>1} 4\cdot {1}/{2}\cdot 2^{-2b}$ and obtain $5/12$ and $13/28$ as the two contributions to $\Dens_{2}(4\alpha)$.

\bibliographystyle{abbrv}
\bibliography{biblio}

\begin{thebibliography}{1}

\bibitem{NathanJones}
N.~Jones.
\newblock Almost all elliptic curves are {S}erre curves.
\newblock {\em Trans. Amer. Math. Soc.}, 362(3):1547--1570, 2010.

\bibitem{JonesRouse}
R.~Jones and J.~Rouse.
\newblock Galois theory of iterated endomorphisms.
\newblock {\em Proc. Lond. Math. Soc. (3)}, 100(3):763--794, 2010.
\newblock Appendix A by Jeffrey D. Achter.

\bibitem{LombardoPeruccaNY}
D.~Lombardo and A.~Perucca.
\newblock The 1-eigenspace for matrices in {$\rm {GL}_2(\mathbb{Z}_\ell)$}.
\newblock {\em New York J. Math.}, 23:897--925, 2017.

\bibitem{LombardoPerucca}
D.~Lombardo and A.~Perucca.
\newblock {Reductions of points on algebraic groups}.
\newblock {\em ArXiv e-prints}, 1612.02847v2, 2017.

\bibitem{PeruccaJNT}
A.~Perucca.
\newblock Prescribing valuations of the order of a point in the reductions of
  abelian varieties and tori.
\newblock {\em J. Number Theory}, 129:469--476, 2009.

\bibitem{Serre}
J.-P. Serre.
\newblock Propri\'et\'es galoisiennes des points d'ordre fini des courbes
  elliptiques.
\newblock {\em Invent. Math.}, 15(4):259--331, 1972.

\bibitem{Sutherland16}
A.~V. Sutherland.
\newblock Computing images of {G}alois representations attached to elliptic
  curves.
\newblock {\em Forum Math. Sigma}, 4:e4, 79, 2016.

\bibitem{lmfdb}
{The LMFDB Collaboration}.
\newblock The {L}-functions and modular forms database.
\newblock \url{http://www.lmfdb.org}, 2016.

\bibitem{sagemath}
{The Sage Developers}.
\newblock {\em {S}ageMath, the {S}age {M}athematics {S}oftware {S}ystem
  ({V}ersion 7.3)}, 2016.
\newblock \url{http://www.sagemath.org}.

\end{thebibliography}
\end{document}